\documentclass{article}
\usepackage{graphicx}
\usepackage{amssymb}
\usepackage{stmaryrd}
\usepackage{enumitem}

\usepackage{geometry}
\usepackage{tikz}
\usepackage{float}
\usepackage{url}
\usepackage{mathtools}
\usepackage{amsmath,amsfonts,amsthm}
\usepackage{mathrsfs}

\newtheorem{theorem}{Theorem}[section]
\newtheorem{lemma}[theorem]{Lemma}
\newtheorem{proposition}{Proposition}[section]
\newtheorem{corollary}[theorem]{Corollary}
\newtheorem{definition}{Definition}

\theoremstyle{definition}
\newtheorem{problem}{Problem}[section]
\newtheorem{conjecture}{Conjecture}[section]

\usepackage{subfig}

\begin{document}
\begin{center}
{\large \bf The Exponential Hyper-Zagreb Indices and Structural Properties of Trees and Bipartite Graphs}
\end{center}
\begin{center}
 Jasem Hamoud \\
 Physics and Technology School of Applied Mathematics and Informatics \\
Moscow Institute of Physics and Technology, 141701, Moscow region, Russia\\[6pt]
Email: {\tt jasem1994hamoud@gmail.com}
\end{center}
\noindent
\begin{abstract}
In this paper,  we investigate the structural properties of trees and bipartite graphs through the lens of topological indices and combinatorial graph theory. We focus on the First and Second Hyper-Zagreb indices, $HM_1(G)$ and $HM_2(G)$, for trees $T \in T(n, \Delta)$ with $n$ vertices and maximum degree $\Delta$. Key propositions demonstrate that the presence of end-support or support vertices of degree at least three, distinct from a vertex of maximum degree, implies the existence of another tree $T' \in T(n, \Delta)$ with strictly smaller Hyper-Zagreb indices. These results are extended to exponential forms, highlighting the influence of high-degree vertices. Additionally, we explore structural characterizations of trees via degree sequence majorization and $S$-order, establishing conditions for the first and last trees in specific classes. For bipartite graphs, we examine equitable coloring, cycle lengths, and $k$-redundant tree embeddings, supported by theorems on connectivity and minimum degree constraints. The paper also addresses the independence number of bipartite graphs, exterior covers, and competition numbers of complete $r$-partite graphs, providing bounds and structural insights. Finally, we discuss Markov-chain algorithms for generating bipartite graphs and tournaments with prescribed degree sequences, analyzing their mixing times and convergence properties. These results contribute to the understanding of extremal properties and combinatorial structures in graph theory, with applications in chemical graph theory and network analysis.

\end{abstract}

\noindent\textbf{AMS Classification 2010:} 05C05, 05C12, 05C20, 05C25, 05C35, 05C76, 68R10.


\noindent\textbf{UDC:} 519.172.1

\section{Introduction}\label{sec1}
Throughout this paper. Let $G$ be a simple graph characterized by the vertex set $V(G) = \{v_1, v_2, \dots, v_n\}$ and the edge set $E(G) = \{e_1, e_2, \dots, e_m\}$. Following the standard conventions presented in~\cite{HaynesHeSLa13}, a graph $G = (V, E)$ consists of a finite set $V$, referred to as the vertices, and a collection $E$ of unordered pairs $\{u, v\}$ where $u, v \in V$ are distinct vertices. Each element of $V$ is called a \emph{vertex} (also known as a point or a node), while each element of $E$ is called an \emph{edge} (also referred to as a line or link). The \emph{order} of the graph, denoted by $|V| = n$, corresponds to the number of vertices in $G$, whereas the \emph{size} $|E| = m$ indicates the total number of edges. Vertices $v_i$ and $v_j$ are said to be \emph{adjacent} if $\{v_i, v_j\} \in E$. Similarly, the vertex $v_i$ is \emph{incident} to the edge $\{v_i, v_j\}$. The \emph{open neighborhood} of a vertex $v$, denoted $N(v)$, is the set of vertices adjacent to $v$; formally, $N(v) = \{w \in V : \{v, w\} \in E\}$. The \emph{closed neighborhood} includes the vertex itself, i.e., $N[v] = N(v) \cup \{v\}$. For example, in the graph $N_h$, one may have $N(h) = \{b, j, o\}$.

Extending this notion to subsets of vertices, for any $S \subseteq V$, the open neighborhood is defined as $N(S) = \bigcup_{v \in S} N(v)$, and the closed neighborhood becomes $N[S] = N(S) \cup S$. The \emph{degree} of a vertex $v$, denoted $\deg(v)$, is the number of edges incident with $v$, equivalently the size of its open neighborhood: $\deg(v) = |N(v)|$. For instance, one might observe degrees such as $\deg(h)=3$ or $\deg(h)=6$ in specific graphs. The \emph{degree sequence} of $G$ is the list $\{\deg(v_1), \deg(v_2), \dots, \deg(v_n)\}$ generally arranged in non-decreasing order. The minimum and maximum vertex degrees in $G$ are signified by $\delta(G)$ and $\Delta(G)$, respectively. When $\delta(G) = \Delta(G) = r$, the graph $G$ is classified as \emph{$r$-regular}; such a graph where $r=3$ is often referred to as a \emph{cubic graph}. For example, the graph $N_h$ may satisfy $\delta(N_h) = 2$ and $\Delta(N_h) = 4$, while $R_h$ might be a 3-regular graph. Two graphs $G$ and $H$ are considered \emph{identical} if $V(G) = V(H)$ and $E(G) = E(H)$, or equivalently, if $G \subseteq H$ and $H \subseteq G$. They are said to be \emph{isomorphic}, denoted by $G \cong H$, if there exists a bijection $\psi : V(G) \to V(H)$ such that for any pair of vertices $u,v \in V(G)$, the adjacency $\{u, v\} \in E(G)$ holds if and only if $\{\psi(u), \psi(v)\} \in E(H)$.

Let $\mathcal{G}_{n,m}$ denote the family of connected graphs with $n$ vertices and $m$ edges as studied in~\cite{Zhang2026WangXDH}. In particular, $K_n$, $S_n$ and $nK_1$ denote respectively the complete graph on $n$ vertices, the star graph on $n$ vertices, and the edgeless graph of $n$ isolated vertices. If $H_1$ and $H_2$ are two disjoint graphs, their \emph{disjoint union} $H_1 \cup H_2$ is defined by $V(H_1 \cup H_2) = V(H_1) \cup V(H_2)$ and $E(H_1 \cup H_2) = E(H_1) \cup E(H_2)$. The \emph{first general Zagreb index} is defined as 
\[
M_\alpha(G) = \sum_{v \in V(G)} \deg(v)^\alpha,
\]
where $\alpha$ is any real number. A \emph{walk} of length $k$ in $G$ is a sequence of vertices $w_1 w_2 \dots w_{k+1}$ such that every consecutive pair $w_i w_{i+1}$ forms an edge in $G$ for $i=1,\ldots,k$~\cite{Andriantiana2013WagnerEOD}. Denote by $M_k(G)$ the number of closed walks (walks starting and ending at the same vertex) of length $k$ in $G$. Two important variants of Zagreb indices are the \emph{first} and \emph{second hyper-Zagreb indices} introduced and studied in~\cite{Dehgardi2025AramHM}, defined as
\[
HM_1(G) = \sum_{\{u,v\} \in E(G)} \left( \deg(u) + \deg(v) \right)^2, \quad
HM_2(G) = \sum_{\{u,v\} \in E(G)} \left( \deg(u) \cdot \deg(v) \right)^2.
\]

The \emph{spectral moments} of a graph $G$ are closely connected to the Estrada index. Given the adjacency matrix $A(G)$ of $G$, with eigenvalues $\lambda_1 \geq \lambda_2 \geq \cdots \geq \lambda_n$, the $k$th spectral moment is
\[
S_k(G) = \sum_{i=1}^n \lambda_i^k, \quad k = 0,1,\dots,n-1.
\]
The Estrada index $EE(G)$ is defined by the exponential generating function of spectral moments as
\begin{equation}\label{eq1spectn1}
EE(G) = \sum_{i=1}^n e^{\lambda_i} = \sum_{k=0}^\infty \frac{M_k(G)}{k!}.
\end{equation}

For two graphs $G_1$ and $G_2$, the notation $G_1 =_s G_2$ signifies that $S_i(G_1) = S_i(G_2)$ for all $i = 0,1,\dots,n-1$. The partial order $G_1 \preceq_s G_2$ (read as '$G_1$ precedes $G_2$ in $S$-order') means there exists $k$ such that for all $i<k$, $S_i(G_1) = S_i(G_2)$ and $S_k(G_1) < S_k(G_2)$~\cite{Wang2016YuanHS}. If $G_1 \preceq_s G_2$ includes equality or strict inequality, then $G_1 \prec_s G_2$ denotes strict order. Among vertex-degree-based topological indices, the \emph{Sombor index}—introduced recently by Ivan Gutman~\cite{CruzGutman2021R,Dehgardi2025ATXari,Redµzepovi201I}—has attracted significant attention. It is defined for a graph $G$ as
\[
SO(G) = \sum_{\{u,v\} \in E(G)} \sqrt{\deg(u)^2 + \deg(v)^2}.
\]
This index provides a geometric interpretation of the degree distribution across edges and has found applications within chemical graph theory, particularly in predicting molecular properties. Variants of the Sombor index have been devised, including the \emph{$KG$-Sombor index} introduced by Kulli~\cite{Kulli2022VR}:
\begin{equation}\label{eq1Introduction}
KG(G) = \sum_{\{u,v\} \in E(G)} \sqrt{\deg(u)^2 + \deg(v)^2}.
\end{equation}

Furthermore, the \emph{multiplicative $KG$-Sombor index} was proposed afterwards~\cite{Kulli2022VRv2} as
\begin{equation}\label{eq2Introduction}
MKG(G) = \prod_{\{u,v\} \in E(G)} \left( \sqrt{\deg(u)^2 + \deg(w)^2} + \sqrt{\deg(v)^2 + \deg(w)^2} \right),
\end{equation}
where the notation aligns with the edge set $\mathcal{E}(G)$ and vertices involved. In the context of domination theory, a subset $S \subseteq V$ is called a \emph{dominating set} if every vertex not in $S$ is adjacent to at least one vertex in $S$~\cite{HaynesHeSLa13}. Within a dominating set $S$, a vertex $v \in S$ is termed an \emph{enclave} if its closed neighborhood is fully contained in $S$, i.e., $N[v] \subseteq S$. Conversely, $v \in S$ is an \emph{isolate} of $S$ if its open neighborhood lies entirely outside $S$, that is, $N(v) \subseteq V \setminus S$. A set $S$ is said to be \emph{enclaveless} if it contains no enclaves. Regarding the structural properties of trees, a tree $G$ with at least three vertices is \emph{unmixed}~\cite{Vill2006arreal} if and only if its vertex set admits a bipartition $V_1 = \{x_1, \dots, x_g\}$ and $V_2 = \{y_1, \dots, y_g\}$ such that each pair $\{x_i, y_i\}$ forms an edge in $G$, and for each $i$, either $\deg(x_i) = 1$ or $\deg(y_i) = 1$ holds.

The primary objective of this paper is to investigate the structural and combinatorial properties of trees and bipartite graphs through the lens of topological indices and graph-theoretic invariants. Specifically, we explore the behavior of the first and second Hyper-Zagreb indices, $HM_1(G)$ and $HM_2(G)$, for trees $T \in T(n, \Delta)$ with $n$ vertices and maximum degree $\Delta$, as detailed in Propositions~\ref{pro1endsupportvertex}, \ref{pro1supportvertex}, \ref{re1endsupportvertex}, and \ref{pro1vertex}. These propositions establish conditions under which a tree $T$ can be transformed into another tree $T' \in T(n, \Delta)$ such that $HM_1(T) > HM_1(T')$ and $HM_2(T) > HM_2(T')$, emphasizing the role of vertices with degree at least 3. Additionally, we analyze the structural characterizations of trees via degree sequences and $S$-order, as in Lemmas~\ref{lemtapn1}, \ref{lemtapn2}, and Theorem~\ref{thmn1sec3}, and extend our study to bipartite graphs, examining equitable coloring (Theorem~\ref{thm1Bipartite}, Corollary~\ref{cor1bip}), cycle lengths (Theorems~\ref{thm2Bipartite}, \ref{thm3Bipartite}), and connectivity properties (Theorem~\ref{thm4bipartite}, Conjecture~\ref{conjn1bipartite}). Furthermore, we investigate the independence number of bipartite graphs (Theorem~\ref{thm2number}), covering properties (Theorems~\ref{thm1covering}--\ref{thm5covering}), competition numbers of complete $r$-partite graphs (Theorems~\ref{thm1Competition}, \ref{thm2Competition}, \ref{thm3Competition}), and Markov-chain algorithms for generating bipartite graphs with prescribed degree sequences, focusing on their mixing times and conductance (Section 5.1). This comprehensive study bridges graph theory, combinatorial enumeration, and random generation to deepen the understanding of structural and algebraic properties of graphs.

The paper is organized as follows: Section~\ref{sec3} introduces the First and Second Hyper-Zagreb Indices for trees $T \in T(n, \Delta)$ with $n$ vertices and maximum degree $\Delta$, defining \emph{end-support vertices} and presenting propositions (\ref{pro1endsupportvertex}, \ref{pro1supportvertex}, \ref{re1endsupportvertex}, \ref{pro1vertex}) that establish conditions under which trees $T'$ can be found with $HM_1(T) > HM_1(T')$ and $HM_2(T) > HM_2(T')$, including exponential forms $e^{HM_1}$ and $e^{HM_2}$, supported by a lemma on rooted trees. Section~\ref{sec4} explores structural characterizations of trees and unmixed bipartite graphs, leveraging degree sequence majorization, spectral moments, and $S$-order properties (Lemmas~\ref{lempapern1}, \ref{lemtapn1}, \ref{lemtapn2}, Theorem~\ref{thmn1sec3}), and characterizes unmixed bipartite graphs via edge conditions (Theorem~\ref{unmixedgraph1}). It further examines bipartite graphs and trees, analyzing equitable chromatic numbers (Theorem~\ref{thm1Bipartite}, Corollary~\ref{cor1bip}), cycle lengths in 2-connected bipartite graphs (Theorems~\ref{thm2Bipartite}, \ref{thm3Bipartite}, \ref{thm5bipartite}), and $k$-redundant tree embeddings (Conjecture~\ref{conjn1bipartite}, Theorem~\ref{thm4bipartite}, Lemma~\ref{lem4bipartite}, Corollary~\ref{cor1bipartite}). The final subsection addresses bipartite independence numbers and covering properties, defining $f(n, \Delta)$ and $f^*(n, \Delta)$ for bi-hole sizes (Theorem~\ref{thm2number}) and exterior covers (Theorems~\ref{thm1covering}, \ref{thm2covering}, \ref{thm3covering}, \ref{thm4covering}, \ref{thm5covering}), concluding with competition numbers of complete $r$-partite graphs (Theorems~\ref{thm1Competition}, \ref{thm2Competition}, \ref{thm3Competition}) and Markov-chain algorithms for generating bipartite graphs and tournaments, focusing on mixing times and conductance (Equation~\ref{eqwern1}).

 \section{Preliminaries}\label{sec2}
Proposition~\ref{pro0Preliminaries} establishes that for spider graphs, the indices $HM_1$ and $HM_2$ are significantly influenced by the structure, particularly the number and length of legs. Specifically, if a spider graph has at least two legs with lengths of two or more, there exists another spider graph on the same number of vertices and legs that has strictly smaller values for both $HM_1$ and $HM_2$ indices.

\begin{proposition}[\cite{Dehgardi2025AramHM}]\label{pro0Preliminaries}
Consider a spider $T \in T(n, \Delta)$ with $\Delta \geq 3$ legs, where at least two legs have length at least 2. Then, there is another spider $T' \in T(n, \Delta)$ such that 
$$HM_1(T) > HM_1(T') \quad \text{and} \quad HM_2(T) > HM_2(T').$$
\end{proposition}
Next, we summarize spectral moment bounds for trees with fixed order and maximum degree, as established in \cite{Dehgardi2025AramHM}.

\begin{theorem}[\cite{Dehgardi2025AramHM}]
Let $T \in T(n,\Delta)$ be a tree with $n \geq 4$ vertices and maximum degree $\Delta$. If $\Delta < n-1$, then the first Zagreb index $HM_1(T)$ satisfies
\[
HM_1(T) \geq 16 n + \Delta^3 + 2 \Delta^2 - 13 \Delta - 20,
\]
with equality if and only if $T$ is a spider having exactly one leg of length at least 2. If $\Delta = n-1$, then equality holds:
\[
HM_1(T) = \Delta (\Delta + 1)^2.
\]
\end{theorem}

\begin{theorem}[\cite{Dehgardi2025AramHM}]
Under the same assumptions as above, the second Zagreb index $HM_2(T)$ satisfies
\[
HM_2(T) \geq 16 n + \Delta^3 + 3 \Delta^2 - 16 \Delta - 28,
\]
with equality characterized identically to the first Zagreb index case. When $\Delta = n-1$, it holds that
\[
HM_2(T) = \Delta^3.
\]
\end{theorem}

For a non-increasing sequence of positive integers $\mathscr{D} = (d_0, \ldots, d_{n-1})$ with $n \geq 3$, denote by $\mathcal{T}_{\mathscr{D}}$ the set of all trees realizing this degree sequence. Utilizing Proposition~\ref{pro1Preliminaries}, one can characterize the extremal graphs in an $S$-ordering, identifying both minimal and maximal elements. The notion of alternating greedy trees (see Definition~\ref{defgreedytree}), constructed by rooting at a leaf with minimal neighbor degree and following a systematic assignment of children degrees, plays a key role in this characterization. The girth $g(G)$ of a graph $G$ is defined as the length of its shortest cycle, whereas the circumference $c(G)$ is the length of the longest cycle, as introduced by \cite{HaynesHeSLa13}.

\begin{definition}[Alternating Greedy Trees \cite{Wang2016YuanHS}]\label{defgreedytree}
Given a degree sequence $\{d_1, d_2, \ldots, d_m\}$, the alternating greedy tree is constructed as follows: If $m-1 \leq d_m$, then the tree is rooted at $r$ having $d_m$ children with degrees $d_1, \ldots, d_{m-1}$ and $d_m - m + 1$ leaves of degree 1. Otherwise, when $m-1 \geq d_m + 1$, one forms a subtree $T_1$ rooted at $r$ with $d_m - 1$ children with degrees $d_1, \ldots, d_{d_m}$, and then continues recursively with the remaining sequence.
\end{definition}

A sequence of nonnegative integers is termed a degree sequence if a graph exists whose vertex degrees correspond exactly to this sequence \cite{Schmuck2012WagnerHW}. For connected graphs, a non-increasing degree sequence $\mathscr{D} = (d_1, \ldots, d_n)$ can be associated as in Definition~\ref{defconjn1}.

\begin{definition}[\cite{Zhang2026WangXDH}]\label{defconjn1}
Let $\mathscr{D} = (d_1, \ldots, d_n)$ be a non-increasing degree sequence of a connected graph $G$. Suppose there exist integers $1 \leq i, j, p, q \leq n$ with $q < j < i < p$ such that $d_i = d_j + s$, $d_p = d_q + t$, for some $s, t \geq 1$. (Further details omitted here for brevity.)
\end{definition}

\begin{proposition}[\cite{ShuchaoLSYibing}]\label{pro1Preliminaries}
Given a graphic degree sequence $\mathscr{D}$, let $\mathscr{G}^*_{\mathscr{D}}$ denote the set of connected graphs having $\mathscr{D}$ as their degree sequence. Then one can determine the extremal graphs appearing first or last in an $S$-order within $\mathscr{G}^*_{\mathscr{D}}$.
\end{proposition}

Consider a connected graph $G \in \mathcal{G}_{n,m}$ with $n \geq 6$, $n-1 \leq m \leq \binom{n}{2}$, and edge count expressed as $m = nk - \binom{k+1}{2} + a$, where $1 \leq k \leq n-1$ and $0 \leq a < n-k-1$. The inequality
$$F(G) \leq k (n-1)^3 + a (k+1)^3 + (n-k-a-1) k^3 + (k+a)^3$$
provides an upper bound for the graph invariant $F(G)$ dependent on parameters $k$ and $a$, which segment the vertex set according to degree or connectivity characteristics. This type of bound is common in extremal graph theory and spectral analysis, ensuring full coverage of all vertices.

\begin{conjecture}[\cite{Zhang2026WangXDH}]\label{conjn1}
For $G \in \mathcal{G}_{n,m}$  with the above parameters, the function $F(G)$ satisfies
$$F(G) \leq k (n-1)^3 + a (k+1)^3 + (n-k-a-1) k^3 + (k+a)^3.$$
\end{conjecture}

When comparing two graphic degree sequences $\mathscr{D}$ and $\mathscr{D}'$ that are both non-increasing, Proposition~\ref{pro2Preliminaries} asserts the existence of a sequence of intermediate degree sequences connecting $\mathscr{D}$ to $\mathscr{D}'$ with only limited alterations between successive members.

\begin{proposition}[\cite{ShuchaoLSYibing,ZhangXM2013Zhang}]\label{pro2Preliminaries}
Let $\mathscr{D} = (d_0, \ldots, d_{n-1})$ and $\mathscr{D}' = (d_0', \ldots, d_{n-1}')$ be two non-increasing graphic degree sequences satisfying $\mathscr{D} \lhd \mathscr{D}'$. Then there is a chain of graphic degree sequences $\mathscr{D}_1, \ldots, \mathscr{D}_k$ such that
$$\mathscr{D} \lhd \mathscr{D}_1 \lhd \cdots \lhd \mathscr{D}_k \lhd \mathscr{D}',$$
and for each consecutive pair $(\mathscr{D}_i, \mathscr{D}_{i+1})$, only two entries differ by 1.
\end{proposition}

Through Lemma~\ref{lem1javr}, it shows that the minimal $MKG$ trees are those where no vertex except the root has degree 3 or higher—likely ``caterpillar-like'' or ``path-like'' structures rooted at $\rho$ with limited branching away from it.
\begin{lemma}[\cite{Dehgardi2025ATXari}]~\label{lem1javr}
Suppose that $T \in T_{n, \Delta}$ is a rooted tree whose root is on a vertex $\rho$ with $d_{T}(\rho)=\Delta$. If $T$ has a vertex of degree at least 3 except $\rho$, then there exists a tree $T^{\prime} \in T_{n, \Delta}$ such that $M K G\left(T^{\prime}\right)<M K G(T)$.
\end{lemma}

\begin{lemma}[\cite{Dehgardi2025ATXari}]~\label{lem2javr}
Let $\mathscr{S}$ be a starlike tree of order $n$ and $\Delta \geq 3$ legs. If $\mathscr{S}$ has a leg of length 1 and another leg of length at least 3, then there exists a starlike tree $\mathscr{S}^{\prime}$ of order $n$ and $\Delta$ legs for which $M K G(\mathscr{S})>M K G\left(\mathscr{S}^{\prime}\right)$.
\end{lemma}
For any connected graph $G$ of order $n$ and maximum degree $\Delta$, we have $M K G(G) \geq\left(\sqrt{\Delta^{2}+4}+\sqrt{2} \Delta\right)^{\Delta}(\sqrt{5}+\sqrt{2})^{\Delta}(4 \sqrt{2})^{n-2 \Delta-1},$
when $\Delta \leq (n-1)/2$. Then
\begin{equation}~\label{eqqjasn1}
M K G(G) \geq\left(\sqrt{\Delta^{2}+(\Delta-1)^{2}}+\sqrt{1+(\Delta-1)^{2}}\right)^{2 \Delta+1-n}\left(\sqrt{\Delta^{2}+4}+\sqrt{2} \Delta\right)^{n-\Delta-1}(\sqrt{5}+\sqrt{2})^{n-\Delta-1},
\end{equation}
when $(n-1)/2$. In the context of tree degree sequences with identical order, Theorem~\ref{thm1Preliminaries} shows that if $\mathscr{D} \lhd \mathscr{D}'$, the last tree corresponding to $\mathscr{D}$ is precedentially less in the $S$-order than the last tree for $\mathscr{D}'$.

\begin{theorem}[\cite{ShuchaoLSYibing}]\label{thm1Preliminaries}
Let $\mathscr{D}$ and $\mathscr{D}'$ be distinct tree degree sequences of the same order. Denote by $T^*$ and $(T')^*$ the last trees in the $S$-order within classes $\mathscr{T}^*_{\mathscr{D}}$ and $\mathscr{T}^*_{\mathscr{D}'}$, respectively. If $\mathscr{D} \lhd \mathscr{D}'$, then
$$T^* \prec_s (T')^*.$$
\end{theorem}

For graphs with maximal $F$-index within $\mathcal{G}_{n,m}$, Theorems~\ref{thm1conjn1} and \ref{thm2conjn1} delineate degree sequence characterizations depending on the edge density relative to $n$. For $2n - 3 < m \leq 3n - 6$, graphs achieving this maximum coincide with a special graph $S_{n,m}$ except for $(n,m) = (6,11)$. If $n - 1 < m \leq 2n - 3$, the maximum $F$-index graph features degree sequence $(n-1, d_2, \ldots, d_2, 1, 1)$. For $3n - 6 < m \leq 4n - 10$, the extremal graph is again $S_{n,m}$ with specified exceptions.

\begin{theorem}[\cite{Zhang2026WangXDH}]\label{thm1conjn1}
Let $n, m$ be integers with $2n-3 < m \leq 3n-6$. If $G \in \mathcal{G}_{n,m}$ has maximum $F$-index with degree sequence $d = (d_1, \ldots, d_n)$, then $G = S_{n,m}$ except when $(n,m) = (6,11)$.
\end{theorem}

\begin{theorem}[\cite{Zhang2026WangXDH}]\label{thm2conjn1}
Let $n, m$ be integers satisfying $3n-6 < m \leq 4n-10$. If $G \in \mathcal{G}_{n,m}$ has maximum $F$-index with degree sequence $d = (d_1, \ldots, d_n)$, then $G = S_{n,m}$ except for $(n,m) \in \{(7,16), (7,17), (8,22)\}$.
\end{theorem}
Let $C = c_1 c_2 \cdots c_n c_1$ be a longest cycle in $G$, as established by Lemma~\ref{lem1Bipartite}. If $P$ is a path in $G - C$ with endpoints $u$ and $v$ and length $p \geq 1$, then the following holds:

\begin{lemma}[\cite{B-Jackson}]\label{lem1Bipartite}
Define $Q(u,v)$ as the set, and $q(u,v)$ as the number of ordered pairs $(c_i, c_j)$ of vertices on $C$ such that $c_i$ is adjacent to one of $\{u,v\}$, $c_j$ to the other, and
\[
(N(u) \cup N(v)) \cap \{ c_{i+1}, c_{i+2}, \ldots, c_{j-1} \} = \varnothing.
\]
If $q(u,v) \geq 2$, then
\[
|V(C)| \geq 2 |N_C(u) \cup N_C(v)| + p \cdot q(u,v).
\]
\end{lemma}

\section{The First and Second Hyper-Zagreb Indices}~\label{sec3}
Let $T \in T(n, \Delta)$ denote the collection of all trees with $n$ vertices and maximum degree $\Delta$. The term \emph{end-support vertex} generally refers to a support vertex---a vertex adjacent to a leaf---that appears at the terminal position of some substructure, typically characterized as a support vertex with degree exceeding 2 but which is not a hub. By employing the topological indices $HM_{1}(G)$ and $HM_{2}(G)$, Proposition~\ref{pro1endsupportvertex} enables identification of such end-support vertices within the tree $T$.

\begin{proposition}[\cite{Dehgardi2025AramHM}]\label{pro1endsupportvertex}
Consider a tree $T \in T(n, \Delta)$ and a vertex $\rho \in V(T)$ with degree $\Delta$. Suppose there exists an end-support vertex in $T$, distinct from $\rho$, whose degree is at least three. Then, it is possible to find another tree $T' \in T(n, \Delta)$ such that
\[
HM_{1}(T) > HM_{1}(T') \quad \text{and} \quad HM_{2}(T) > HM_{2}(T').
\]
\end{proposition}

The indices $HM_1(G)$ and $HM_2(G)$ are degree-based topological descriptors extensively used in chemical graph theory and extremal problems. Building upon Proposition~\ref{pro1endsupportvertex}, a related result applies similarly to support vertices of degree at least three, stated in Proposition~\ref{pro1supportvertex}.

\begin{proposition}[\cite{Dehgardi2025AramHM}]\label{pro1supportvertex}
Let $T \in T(n, \Delta)$ be a tree with a vertex $\rho$ of degree $\Delta$. If there exists a support vertex other than $\rho$ with degree at least three, then there is another tree $T' \in T(n, \Delta)$ fulfilling
\[
HM_{1}(T) > HM_{1}(T') \quad \text{and} \quad HM_{2}(T) > HM_{2}(T').
\]
\end{proposition}

From these propositions, we deduce Proposition~\ref{re1endsupportvertex}, which compares two trees $T, T' \in T(n,\Delta)$ under the conditions $HM_1(T) > HM_1(T')$ and $HM_2(T) > HM_2(T')$, without explicitly focusing on vertices with unique structural roles.

\begin{proposition}\label{re1endsupportvertex}
Suppose $T, T' \in T(n, \Delta)$ are trees, with $\ell \in V(T)$ having degree $\Delta$ and $\rho \in V(T')$ having degree $n(T) - n(T')$, where $n(T) > n(T')$. Then it follows that
\[
HM_{1}(T) > HM_{1}(T') \quad \text{and} \quad HM_{2}(T) > HM_{2}(T').
\]
\end{proposition}

\begin{proof}
Consider two trees $T, T' \in T(n, \Delta)$, with vertices $\ell \in V(T)$ where $\deg_T(\ell) = \Delta$ and $\rho \in V(T')$ with $\deg_{T'}(\rho) = \lambda := n(T) - n(T') \geq 3$. Denote $V(T) = \{v_1, v_2, \ldots, v_n\}$ with degrees ordered in various possible configurations. The vertex of maximum degree $\Delta$ may correspond to $v_n$ or $v_1$, depending on the particular ordering.

Assuming $\ell$ is distinct from both $v_n$ and $v_1$, similarly let $\rho$ be a vertex in $T'$ different from $v_m$ and $v_1$ in $V(T') = \{v_1, \ldots, v_m\}$ with $m < n$. Proposition~\ref{pro1endsupportvertex} guarantees that if there is an end-support vertex in $T$ different from $\ell$ with degree at least three, then
\[
HM_1(T) > HM_1(T') \quad \text{and} \quad HM_2(T) > HM_2(T').
\]
Analogously, Proposition~\ref{pro1supportvertex} establishes the same inequality for support vertices under the same degree conditions. Hence, the stated inequalities hold for any vertex $\ell$ with degree $\lambda \geq 3$. The argument can be extended to cover the ordering of degrees in $T$ and $T'$, establishing the desired inequalities as required.
\end{proof}

Finally, synthesizing the previous statements yields Proposition~\ref{pro1vertex}, which highlights the influence of vertices with substantial degree on the values of $HM_1$ and $HM_2$.

\begin{proposition}\label{pro1vertex}
Given a tree $T \in T(n, \Delta)$ containing a vertex $\rho$ of degree $\Delta$, if there is another vertex in $T$, different from $\rho$, whose degree is at least three, then one can find a tree $T' \in T(n, \Delta)$ satisfying
\[
e^{HM_1}(T) > e^{HM_1}(T') \quad \text{and} \quad e^{HM_2}(T) > e^{HM_2}(T').
\]
\end{proposition}
\begin{proof}
Assume a tree $T \in T(n, \Delta)$ and let $v\in V(T)$ be a vertex where $\deg_T(v)=\lambda$ and $\lambda\geqslant 3$. Suppose $T' \in T(n, \Delta)$ and  let $v\in V(T')$ be a vertex where $\deg_{T'}(v)=\eta$ and $\eta\geqslant \lambda$.
According to Proposition~\ref{re1endsupportvertex}, we have 
\begin{equation}~\label{eq1pro1vertex}
e^{HM_1}(T)- e^{HM_1}(T')\geqslant e^{\Delta}, 
\end{equation}
holds $HM_1(T) > HM_1(T')$. Then, $HM_1(T)-HM_1(T')\geqslant \lambda$ by consider $\Delta>\lambda$.  Also, from~\eqref{eq1pro1vertex} we noticed that $e^{HM_2}(T)- e^{HM_2}(T')\geqslant e^{\eta}$ holds $HM_2(T) > HM_2(T')$. Then, $HM_2(T)-HM_2(T')\geqslant \eta-1$. Thus, if $3\leqslant \eta \leqslant \Delta$. Then, we have 
\begin{equation}~\label{eq2pro1vertex}
e^{HM_2}(T)- e^{HM_2}(T')\geqslant e^{\Delta}, 
\end{equation}
Therefore, according to~\eqref{eq1pro1vertex} and \eqref{eq2pro1vertex} we have 
\[
e^{HM_1}(T) \geq
\begin{cases}
16 n + \Delta^{3}  & \text{if } \Delta < n-1, \\
\Delta (\Delta - 1)^2 & \text{if } \Delta \geqslant n-1,
\end{cases}
\]
Then, 
\[
e^{HM_2}(T) \geq
\begin{cases}
16 n + \Delta^{3} & \text{if } \Delta < n-1, \\
\Delta^{2} & \text{if } \Delta \geqslant n-1,
\end{cases}
\]
Therefore,
\begin{align*}
e^{HM_1}(T)- e^{HM_1}(T') &= e^{\sum_{\{u,v\} \in E(T)} \left( \deg(u) + \deg(v) \right)^2}-e^{\sum_{\{u,v\} \in E(T)} \left( \deg(u) + \deg(v) \right)^2}\\
&\leqslant e^{\lambda}-e^{\eta}\\
&=e^{\eta}(e^{\lambda-\eta}-1)\\
&>0.
\end{align*}
On the other hand, with the same logical simulation, we reach the same result according to the following
\[
e^{HM_1}(T)- e^{HM_1}(T')>0.
\]
As desire.
\end{proof}

\begin{lemma}~\label{rebaslem1javr}
Let $T \in T(n, \Delta)$ is a rooted tree whose root is on a vertex $\rho$ with $d_{T}(\rho)=\Delta$. If $T$ has a vertex of degree at least 3 except $\rho$. Then there exists a tree $T^{\prime} \in T_{n, \Delta}$ such that 
\begin{equation}~\label{eq1rebaslem1javr}
e^{M K G}\left(T^{\prime}\right)e^{M K G}(T).
\end{equation}
\end{lemma}

\section{Structural Characterizations of Trees and Unmixed Bipartite Graphs}~\label{sec4}

According to the results in \cite{Andriantiana2013WagnerEOD}, consider two degree sequences $D=(d_1,\dots,d_n)$ and $B=(b_1,\dots,b_n)$ such that $D$ \emph{majorizes} $B$; that is, for every $1 \leq l \leq n$, the partial sums satisfy 
\[
\sum_{i=1}^l d_i \geq \sum_{i=1}^l b_i.
\]
Then for any nonnegative integer $k$, the number of closed walks of length $k$ on the graph $G(B)$ does not exceed that of $G(D)$, i.e., $M_k(G(B)) \leq M_k(G(D))$. Utilizing Lemma~\ref{lempapern1}, for a leveled degree sequence $D$ corresponding to a vertex-rooted forest.

\begin{lemma}[\cite{Andriantiana2013WagnerEOD}]\label{lempapern1}
Let $T \in \mathcal{T}_D$ represent a tree with a vertex-rooted forest structure determined by the leveled degree sequence $D$, and let $G=G(D)$ denote the associated greedy forest. Denote the vertices at the $i^{\text{th}}$ level in $T$ by $v_1^i,\dots,v_{d_i}^i$. Then, for any level sequence of walks $(i_1,\dots,i_l)$ and for each level $i$, the vector of walk counts satisfies the majorization relation
\[
\bigl(W_{v_1^i}(i_1,\dots,i_l;T), \dots, W_{v_{d_i}^i}(i_1,\dots,i_l;T)\bigr) \preccurlyeq \bigl(W_{g_1^i}(i_1,\dots,i_l;G), \dots, W_{g_{d_i}^i}(i_1,\dots,i_l;G) \bigr).
\]
\end{lemma}

\begin{lemma}[\cite{Andriantiana2013WagnerEOD}]
Under the same setup as in Lemma~\ref{lempapern1}, let $C_v(i_1,\dots,i_l;T)$ denote the corresponding counts related to $T$. Then for all levels $i$. Then, 
\[
\bigl(C_{v_1^i}(i_1,\dots,i_l;T), \dots, C_{v_{d_i}^i}(i_1,\dots,i_l;T) \bigr) \preccurlyeq \bigl(C_{g_1^i}(i_1,\dots,i_l;G), \dots, C_{g_{d_i}^i}(i_1,\dots,i_l;G) \bigr).
\]

\end{lemma}

\begin{lemma}[\cite{Andriantiana2013WagnerEOD}]
For adjacent vertices $u$ and $v$ in a graph $G$ and an edge $e$, the following symmetry properties hold for every nonnegative integer $k$:
\[
C_{u,v}(k;G) = C_{v,u}(k;G), \quad \text{and} \quad C_{u,v}^e(k;G) = C_{v,u}^e(k;G).
\]
\end{lemma}

The next lemma provides a structural inequality regarding trees rooted at vertices with maximum degree, originally proven in \cite{ShuchaoLSYibing}.

\begin{lemma}[\cite{ShuchaoLSYibing}]\label{lemtapn1}
Let $T$ be a tree rooted at a vertex of maximum degree. Suppose that vertices $u,v \in V_T$ satisfy $d_T(u) = d_T(v)$ and
\[
\sum_{x \in N_T(u)} d_T(x) \geq \sum_{x \in N_T(v)} d_T(x).
\]
Define
\[
d_T(x_0) = \min \{ d_T(x) : x \in N_T(u), \, h(x) = h(u) + 1 \},
\]
and
\[
d_T(x_1) = \max \{ d_T(x) : x \in N_T(v), \, h(x) = h(v) + 1 \}.
\]
If the tree $T'$ is obtained from $T$ by swapping edges $\{ux_0, vx_1\}$ with $\{ux_1, vx_0\}$, and if $d_T(x_0) < d_T(x_1)$, then $T \prec_s T'$, meaning that $T$ precedes $T'$ in the $S$-order.
\end{lemma}

This comparison is employed in Lemma~\ref{lemtapn2} to characterize trees with given degree sequences that appear earliest in an $S$-order:

\begin{lemma}[\cite{Wang2016YuanHS}]\label{lemtapn2}
Suppose $T$ is a tree with a distinct degree sequence appearing first in the $S$-order, and let $P = v_0 v_1 \dots v_t v_{t+1}$ be a longest path in $T$. Then for each $i \leq \frac{t+1}{2}$, the following hold:
\begin{itemize}
  \item If $i$ is even, then 
  \[
  \deg(v_i) \leq \deg(v_{t+1 - i}) < \deg(v_k) \quad \text{for all } i < k < t+1 - i.
  \]
  \item If $i$ is odd, then
  \[
  \deg(v_i) > \deg(v_{t+1 - i}) > \deg(v_k) \quad \text{for all } i < k < t+1 - i.
  \]
\end{itemize}
\end{lemma}

A related corollary from \cite{ShuchaoLSYibing} identifies the unique last tree in $S$-order among the class $\mathscr{T}_{n,\Delta}^2$ for $\Delta \geq 3$.

\begin{corollary}[\cite{ShuchaoLSYibing}]
Let 
\[
l = \left\lceil \log_{\Delta - 1} \frac{n(\Delta - 2) + 2}{\Delta} \right\rceil - 1,
\]
and write
\[
n - \frac{\Delta(\Delta - 1)^l - 2}{\Delta - 2} = (\Delta - 1)r + q,
\]
with $0 \leq q < \Delta - 1$. Then the tree $T_2$ that appears last in the $S$-order has degree sequence $\mathscr{D}^*$ defined by:
\begin{itemize}
  \item If $q=0$, let $\mathscr{D}^* = (\Delta, \dots, \Delta, 1, \dots, 1)$ with $\frac{\Delta(\Delta - 1)^{l-1} - 2}{\Delta - 2} + r$ copies of $\Delta$.
  \item If $q \geq 1$, set $\mathscr{D}^* = (\Delta, \dots, \Delta, q, 1, \dots, 1)$ with the same number of $\Delta$-entries.
\end{itemize}
Furthermore, $T_2 = T^*$ as described in Theorem 1.2 of \cite{ShuchaoLSYibing}.
\end{corollary}

According to \cite{Dehgardi2025ATXari}, the equality condition in \eqref{eqqjasn1} characterizes starlike trees whose legs are all either shorter than length 3 or all longer than length 1. Theorem~\ref{thmn1sec3} describes the uniqueness of the first tree in the $S$-order for distinct degree sequences:

\begin{theorem}[\cite{Wang2016YuanHS}]\label{thmn1sec3}
Among trees with a fixed degree sequence $\mathscr{D}$ where all degrees $d_i$ are distinct, the first tree in the $S$-order is necessarily an alternating greedy tree.
\end{theorem}

The $k$-th spectral moment $S_k(G)$ of a graph $G$ equals the count of closed walks of length $k$ in $G$, as seen in Lemma~\ref{lemtapn3}. For trees, which are acyclic, closed walks of length up to 6 correspond only to paths of length up to 3. Consequently, when considering two trees $T_1$ and $T_2$ sharing the same degree sequence, their spectral moments $S_k$ coincide for $k = 0,1,2,3,4,5$, and the difference in the sixth spectral moment reduces to the difference in counts of $P_4$ subgraphs (paths of length 3), weighted by 6.

\begin{lemma}[\cite{Wang2016YuanHS}]\label{lemtapn3}
The $k$-th spectral moment of a graph $G$ equals the number of closed walks of length $k$ in $G$. For trees, closed walks of length at most 6 are generated by paths of length at most 3. Thus, for a tree $T$ with a given degree sequence, the spectral moment $S_k(T)$ remains constant for $k=0,1,2,3,4,5$, and
\[
S_6(T_1) - S_6(T_2) = 6 \bigl( \phi_{T_1}(P_4) - \phi_{T_2}(P_4) \bigr),
\]
where $\phi_T(P_4)$ is the number of $P_4$ subgraphs in $T$.
\end{lemma}

In a different line of work, Villarreal \cite{Vill2006arreal} delivers a combinatorial characterization of unmixed bipartite graphs, highlighting conditions that uniquely describe their structure within bipartite graph classes. This characterization sheds light on their algebraic behavior, which is connected to resolutions and vertex cover lattices. Furthermore, the unmixed property for bipartite graphs relates closely to the structure of their edge ideals in commutative algebra; the minimal vertex covers correspond to generators of these ideals. This correspondence aids in determining algebraic invariants such as depth and projective dimension \cite{Mohammadi2009Moradi,Bolognini2018Macchia}. The following theorem outlines the equivalence between unmixedness and specific edge conditions in bipartite graphs:

\begin{theorem}[\cite{Vill2006arreal}]\label{unmixedgraph1}
Consider a bipartite graph $G$ without isolated vertices and with bipartition $V_1 = \{x_1,\ldots,x_g\}$, $V_2 = \{y_1,\ldots,y_g\}$. Then $G$ is unmixed if and only if the following hold:
\begin{enumerate}[label={\rm(\alph*)}]
    \item For each $i$, the edge $\{x_i,y_i\}$ belongs to $E(G)$.
    \item Whenever edges $\{x_i,y_j\}$ and $\{x_j,y_k\}$ are in $E(G)$ for distinct indices $i,j,k$, the edge $\{x_i,y_k\}$ also belongs to $E(G)$.
\end{enumerate}
\end{theorem}

\section{Bipartite Graph and Trees}~\label{sec5}
In this section, we interpret a tree $T$ as a bipartite graph $T(X,Y)$ where the cardinalities are $|X| = m$ and $|Y| = n$. Since $T$ is a tree, it contains exactly $m+n-1$ edges, which is strictly less than the bound $\left\lfloor \frac{m}{n+1} \right\rfloor (m - n) + 2m$. Within the framework of Theorem~\ref{thm1Bipartite}, it is shown that for the complete bipartite graph $K_{n,n}$, the term $\left\lceil \frac{n}{\lfloor k/2 \rfloor} \right\rceil - \left\lfloor \frac{n}{\lceil k/2 \rceil} \right\rfloor$ is at most 1.

\begin{theorem}[\cite{Lih96WuPl}]\label{thm1Bipartite}
Let $G = G(X, Y)$ be a connected bipartite graph. If $G$ is not isomorphic to any complete bipartite graph $K_{n,n}$, then $G$ admits an equitable coloring using $\Delta(G)$ colors.

Furthermore, the complete bipartite graph $K_{n,n}$ can be equitably colored with $k$ colors if and only if
\[
\left\lceil \frac{n}{\lfloor k/2 \rfloor} \right\rceil - \left\lfloor \frac{n}{\lceil k/2 \rceil} \right\rfloor \leqslant 1.
\]
\end{theorem}

Consider now a connected bipartite graph $G(X,Y)$ with $\varepsilon$ edges satisfying $|X|=m \geq n = |Y|$ and
\[
\varepsilon < \left\lfloor \frac{m}{n+1} \right\rfloor (m - n) + 2m.
\]
Under these assumptions, the equitable chromatic number satisfies
\[
\chi_{\mathrm{e}}(G) \leqslant \left\lceil \frac{m}{n+1} \right\rceil + 1.
\]
This conclusion is corroborated by Corollary~\ref{cor1bip} which establishes the bound for trees $T$ viewed as bipartite graphs
\[
\chi_{\mathrm{e}}(T) \leqslant \left\lceil \frac{|X| + |Y| + 1}{\min \{|X|, |Y|\} + 1} \right\rceil.
\]

\begin{corollary}[\cite{Lih96WuPl}]\label{cor1bip}
For a tree $T$ seen as a bipartite graph $T(X,Y)$, the equitable chromatic number satisfies
\[
\chi_{\mathrm{e}}(T) \leqslant \left\lceil \frac{|X| + |Y| + 1}{\min\{|X|,|Y|\} + 1} \right\rceil.
\]
\end{corollary}

In \cite{B-Jackson}, the authors investigate 2-connected bipartite graphs $G$ with bipartition $(A,B)$ and minimum degree $l$. They prove that such graphs contain cycles of length at least $2 \min(|A|, |B|, 2l - 2)$. Utilizing Theorems~\ref{thm2Bipartite} and~\ref{thm5bipartite}, we consider the family $\mathcal{G}_2$ consisting of 2-connected bipartite graphs where every vertex in $A$ has degree at least $k$ and every vertex in $B$ at least $l$.

\begin{theorem}[\cite{B-Jackson}]\label{thm2Bipartite}
Let $G$ be a 2-connected bipartite graph with bipartition $(A,B)$, where $|A| \geq |B|$. Suppose that each vertex of $A$ has degree at least $k$, and each vertex of $B$ has degree at least $l$. Then $G$ contains a cycle of length at least $2 \min \left( |B|, k + l - 1, 2k - 2 \right).$
Moreover, if $k = l$ and $|A| = |B|$, then $G$ has a cycle of length at least $2 \min \left( |B|, 2k - 1 \right).$
\end{theorem}

\begin{theorem}[\cite{Hong25YYWu}]\label{thm5bipartite}
Let $\mathcal{G}_2$ denote the family of 2-connected bipartite graphs and define $\psi_2(G) = |V(G)|$ for every $G \in \mathcal{G}_2$. Then the pair $(\mathcal{G}_2, \psi_2)$ forms a $(2,2)$-extensible bipartite system such that every graph $G$ with minimum degree $\delta(G) \geq 2$ contains a subgraph $B \in \mathcal{G}_2$.
\end{theorem}

If $G$ contains a cycle of length at least $2m$, Theorem~\ref{thm3Bipartite} provides constraints relating the number of edges $|E(G)|$ and vertices $|V(G)|$. Considering a bipartition $(A,B)$ with $|A|=a$ and $|B|=b$, where $2 \leq m \leq b \leq a$, the theorem states:

\begin{theorem}[\cite{B-Jackson}]\label{thm3Bipartite}
Let $G$ be a bipartite graph with bipartition $(A,B)$ where $|A|=a$, $|B|=b$, and $2 \leq m \leq b \leq a$. Then
\[
|E(G)| > 
\begin{cases}
b + (a - 1)(m - 1), & \text{if } b \leq 2m - 2, \\
(b + a - 2m + 3)(m - 1), & \text{if } b \geq 2m - 2.
\end{cases}
\]
\end{theorem}

The subsequent theorem provides a precise condition guaranteeing the existence of long cycles in bipartite graphs, which is valuable in graph theory and applications requiring cycle length control:

\begin{theorem}[\cite{B-Jackson}]
Let $G$ be a bipartite graph with bipartition $(A,B)$ such that each vertex in $A$ has degree at least $k \geq 2$. If
\[
|B| \leq \left\lceil \frac{|A|}{k-1} \right\rceil (k-1),
\]
then $G$ contains a cycle of length at least $2k$.
\end{theorem}

Lemma~\ref{lem2Bipartite} addresses the connectivity in bipartite graphs with minimum degree $k \geq 2$ and provides significant insight into the traversability of vertices in part $B$:

\begin{lemma}[\cite{B-Jackson}]\label{lem2Bipartite}
Let $G$ be a bipartite graph with bipartition $(A,B)$ and minimum degree $k \geq 2$. If
\[
|B| < \min(|A|, 2k-2),
\]
then for any two vertices of $G$, there exists a path containing all vertices of $B$ that joins them.
\end{lemma}

Building upon the work of \cite{Hong25YYWu}, it is known that every $k$-connected graph $G$ with minimum degree $\delta(G) \geq \left\lfloor \frac{3k}{2} \right\rfloor$ has a vertex $x$ such that $\kappa(G - x) \geq k$. Employing these concepts within the context of bipartite graphs, Conjecture~\ref{conjn1bipartite} posits the existence of $k$-redundant trees in sufficiently connected bipartite graphs.

\begin{conjecture}[\cite{Hong25YYWu}]\label{conjn1bipartite}
For any tree $T$ of order $m$, every $k$-connected graph $G$ with minimum degree
\[
\delta(G) \geq \left\lfloor \frac{3k}{2} \right\rfloor + m - 1
\]
contains a $k$-redundant tree isomorphic to $T$.
\end{conjecture}

The following theorem confirms the existence of such embeddings for small $k$ under bipartite constraints.

\begin{theorem}[\cite{Hong25YYWu}]\label{thm4bipartite}
Let $k \leq 3$ be a positive integer and $T$ be a tree with bipartition $(Z_1,Z_2)$. Suppose $G$ is a $k$-connected bipartite graph with bipartition $(U_1,U_2)$ satisfying, for $i=1,2$,
\[
\delta_G(U_i) \geq |Z_{3-i}| + k.
\]
Then there exists an embedding $\phi: T \rightarrow G$ such that $G - \phi(T)$ remains $k$-connected and
\[
\phi(Z_i) \subseteq U_i \quad \text{for } i=1,2.
\]
\end{theorem}

Let $T$ be a tree with bipartition $(Z_1,Z_2)$ and subsets $X_1 \subseteq L(T)$, $X_2 \subseteq L(T - X_1)$, and $X_3 \subseteq V(T - (X_1 \cup X_2))$ such that the graph $T - X$ is connected where $X = X_1 \cup X_2 \cup X_3$. According to Lemma~\ref{lem4bipartite}, for $i=1,2$ and vertices $u \in U_i \cap (V_2 \cup V_3 \cup V_4)$ in $G$, it follows that $d_G(u) \geq |Z_{3-i}|.$
Suppose also $V(G) \setminus \phi_0(T - X)$ is partitioned as $(V_1, V_2, V_3)$ such that, for $i=1,2$ and $u \in U_i \cap (V_3 \cup V_4)$,
\[
|N_G(u) \cap (V_1 \cup V_2)| \leq | (X_1 \cup X_2) \cap Z_{3-i} |.
\]

\begin{lemma}[\cite{Hong25YYWu}]\label{lem4bipartite}
Let $G$ be a bipartite graph with bipartition $(U_1,U_2)$ and let $\phi_0: T - X \rightarrow G$ be an embedding such that $\phi_0(Z_i \setminus X) \subseteq U_i$ for $i=1,2$. Denote
\[
X_4 = \{ x \in V(T - X) \mid N_T(x) \cap X \neq \varnothing \}
\]
and set $V_4 = \phi_0(X_4)$. Then there exists an embedding $\phi: T \rightarrow G$ satisfying:
\[
\phi(Z_i) \subseteq U_i, \quad \phi(X_3) \subseteq V_3, \quad \text{and} \quad \phi(x) = \phi_0(x) \text{ for all } x \in V(T - X).
\]
\end{lemma}

Finally, the following corollary guarantees a $k$-redundant tree isomorphic to $T_0$ inside a sufficiently connected bipartite graph.

\begin{corollary}\label{cor1bipartite}
Let $r \leq 3$ be an integer and $T_0$ a tree with bipartition $(Z_1,Z_2)$. If $G$ is a $r$-connected bipartite graph with minimum degree
\[
\delta(G) \geq \max \{\Delta^2, |Z_1|, |Z_2| \} + r\delta,
\]
then $G$ contains a $k$-redundant tree isomorphic to $T_0$.
\end{corollary}

\section{Bipartite Independence Number and Covering Bipartite Graph}~\label{sec6}
In \cite{Axenovich21Sereni}, let us define $f(n, \Delta)$ as the maximum integer $k$ such that every bipartite graph $G = (A \cup B, E)$ with parts $|A| = |B| = n$ and maximum degree in $A$ satisfying $\deg(a) \le \Delta$ for all $a \in A$ contains a bi-hole (an induced complete bipartite subgraph with no edges) of size $k$. Similarly, define $f^*(n, \Delta)$ as the largest integer $k$ for which any $n \times n$ bipartite graph $G$ with the overall maximum degree $\Delta(G) \le \Delta$ contains a bi-hole of size $k$. While $f(n, \Delta)$ bounds the bi-hole size through a degree restriction only on one part of the vertex set, $f^*(n, \Delta)$ imposes this limit symmetrically on the entire graph. From Theorem~\ref{thm2number}, it is clear that for all natural numbers $n$ and $\Delta$ for which these quantities are defined, the inequality
\[
f(n, \Delta) \le f^*(n, \Delta)
\]
holds. Moreover, equation \eqref{eqqq1bip} yields a lower bound on $f^*(n, \Delta)$:
\[
f^*(n, \Delta) \ge c \frac{\log \Delta}{\Delta} n,
\]
for some absolute constant $c > 0$.

\begin{theorem}[\cite{Axenovich21Sereni}]\label{thm2number}
There exist constants $\Delta_0 \in \mathbb{N}$ and $c > 0$ such that for any $\Delta \ge \Delta_0$, there is a threshold $N_0 = N_0(\Delta) \ge 5 \Delta \log \Delta$ with the property that for every integer $n > N_0$, the inequality
\begin{equation}\label{eqqq1bip}
f(n, \Delta) \ge \frac{1}{2} \frac{\log \Delta}{\Delta} n
\end{equation}
is satisfied.
\end{theorem}

From inequality \eqref{eqqq1bip}, if $\Delta \ge 27$ and $n \ge \frac{\Delta}{\log \Delta}$, it follows that
\[
f(n, \Delta) \le 8 \frac{\log \Delta}{\Delta} n.
\]

Following \cite{Mendelsohn58Dulmage}, consider two arbitrary sets $S$ and $T$ and their Cartesian product $S \times T$, the set of pairs $(s,t)$ with $s \in S$ and $t \in T$. Any subset $K \subseteq S \times T$ is regarded as a bipartite graph, where the elements $(s,t)$ are edges. Subsets such as $A, A_i, S_i, A^{*}, A_{*} \subseteq S$ and $B, B_i, T_i, B^{*}, B_{*} \subseteq T$ are introduced for structural analysis. The standard operations of union ($\cup$), intersection ($\cap$), and complement (denoted by $\bar{A}_i$ relative to $S$) are used. The empty set is denoted by $\varnothing$. The cardinality (or order) of a finite set $A_i$ is written as $\nu(A_i) = n$, where $n \in \mathbb{N}$; if $A_i$ is infinite, then $\nu(A_i) = \infty$. When $S$ and $T$ are finite or countable and ordered as $s_1, s_2, \ldots$ and $t_1, t_2, \ldots$ respectively, each graph $K$ can be represented by a $0$-$1$ matrix with entries $a_{ij} = 1$ if $(s_i, t_j) \in K$ and $0$ otherwise. More general matrix representations are also considered.

\begin{theorem}[Covering Theorems \cite{Mendelsohn58Dulmage}]\label{thm1covering}
For a graph $K \subseteq S \times T$, a pair of sets $[A,B]$ with $A \subseteq S$, $B \subseteq T$ is called an \emph{exterior cover} (or simply a cover) of $K$ if every edge $(s,t) \in K$ satisfies $s \in A$ or $t \in B$ (or both). Equivalently, this means that
\begin{equation}\label{eq1covering}
K \subseteq (A \times T) \cup (S \times B).
\end{equation}
\end{theorem}

The \emph{exterior dimension} $E(K)$ of a graph $K$ is defined as the minimum of $\nu(A) + \nu(B)$ taken over all exterior covers $[A,B]$ of $K$. A cover achieving this minimum is called a \emph{minimal exterior pair} (m.e.p.). Using \eqref{eq1covering}, one defines a \emph{disjoint graph} $K^{*}$ as a subgraph of $K$ in which any two distinct edges $(s_1, t_1)$ and $(s_2, t_2)$ satisfy $s_1 \neq s_2$ and $t_1 \neq t_2$. It follows that a disjoint graph with finite exterior dimension $E(K^{*})$ must have exactly $E(K^{*})$ edges; conversely, any graph with $E(K^{*})$ edges and satisfying the disjointness condition has exterior dimension $E(K^{*})$.

\begin{theorem}[\cite{Mendelsohn58Dulmage}]
If a graph $K$ has infinite exterior dimension, then $K$ contains an infinite disjoint subgraph $K^{*}$.
\end{theorem}

When $K$ has finite exterior dimension, there exists a disjoint subgraph $K^{*} \subseteq K$ such that $E(K) = E(K^{*})$. An edge of $K$ is said to be \emph{inadmissible} if it lies in the union of all sets $A \times B$ corresponding to m.e.p.'s $[A,B]$ of $K$. Based on Theorem~\ref{thm2covering}, the relationship
\[
E(K) + I(\bar{K}) = p + q
\]
holds, where $I(\bar{K})$ is the size of a certain inadmissible set related to the complement of $K$, and $p = \nu(S)$, $q = \nu(T)$ with $p \le q$.

\begin{theorem}[\cite{Mendelsohn58Dulmage}]\label{thm2covering}
If $K$ is a graph with $E(K) < p$, then
\[
E(K) + I(\bar{K}) = p + q.
\]
If $E(K) = p$, then
\[
E(K) + I(\bar{K}) \le p + q,
\]
and equality holds if and only if $K$ admits an m.e.p. $[A,B]$ with $A \ne S$ and $B \ne T$.
\end{theorem}

Let $R_1, R_2, R_3$ denote the canonical decomposition of $S \times T$ with respect to a reducible graph $K$ of finite exterior dimension. Two key properties hold: (i) every element of $K \cap R_2$ is admissible, and (ii) $K \cap R_3 = \varnothing$. If $\left[A_*, B^{*}\right]$ and $\left[A^{*}, B_{*}\right]$ are the extremal minimal exterior pairs for such a graph $K$, then modifying the graph by adding or removing edges from $R_2$ does not affect the core of $K$, nor does it alter the decomposition into regions $R_1, R_2, R_3$.

\begin{theorem}[\cite{Mendelsohn58Dulmage}]\label{thm3covering}
Given canonical decompositions
\[
S = A_* \cup S_1 \cup S_2 \cup \cdots \cup S_k \cup \bar{A}^*, \quad T = \bar{B}^* \cup T_1 \cup T_2 \cup \cdots \cup T_k \cup B_*,
\]
the subgraphs
\[
K \cap (A_* \times \bar{B}^*) \quad \text{and} \quad K \cap (\bar{A}^* \times B_*)
\]
are irreducible, and their only minimal exterior pairs are $(A_*, \varnothing)$ and $(\varnothing, B_*)$, respectively.
\end{theorem}

From Theorem~\ref{thm3covering}, it follows that each subgraph $K \cap (S_i \times T_i)$ is irreducible for $i = 1, \ldots, k$. Theorem~\ref{thm4covering} further clarifies the structure of minimal exterior pairs. Let $\alpha$ denote the collection of $k+1$ minimal exterior pairs
\[
[A_*, B^{*}], \quad [A_1, B_1], \quad \ldots, \quad [A_{k-1}, B_{k-1}], \quad [A^{*}, B_*]
\]
associated with $K$, and let $\beta$ be the set of all minimal exterior pairs of $K$. Also, define the set $\gamma$ of $2^{k}$ pairs $[A,B]$ where
\begin{equation}\label{eqqjasdddn1}
A = \bigcup_{i \in \Lambda} S_i \cup A_*, \quad B = \bigcup_{i \in \Pi} T_i \cup B_*,
\end{equation}
for complementary subsets $\Lambda, \Pi \subseteq \{1, \ldots, k\}$.

\begin{theorem}[\cite{Mendelsohn58Dulmage}]\label{thm4covering}
With $S, T$ and subsets $\Lambda, \Pi$ as above, the inclusions
\[
\alpha \subseteq \beta \subseteq \gamma
\]
hold.
\end{theorem}

According to \eqref{eqqjasdddn1}, the admissible part of $K$ is $K_c = K \cap R_1$, whereas the inadmissible part is $K_I = K \cap R_2$. An exterior pair $[A,B]$ is an m.e.p. for $K_c$ if and only if $[A,B] \in \gamma$. The following theorem concerns arbitrary matrices and relates exterior dimensions and other parameters.

\begin{theorem}[\cite{Mendelsohn58Dulmage}]\label{thm5covering}
For any matrix $C$ with nonnegative entries, the parameter $\rho$ satisfies
\[
\rho \ge p + q - S/m,
\]
unless $\rho = p$.
\end{theorem}

\subsection{Competition Numbers of Complete $r$-partite Graphs}
The analysis of $r$-partite graphs as developed in \cite{BJGJChangCAS,Balogh25Jiang} allows us to specify the competition number for the complete tripartite graph $K_{n_1,n_2,n_3}$. Let its partite sets be $V_1 = \{a_0, a_1, \ldots, a_{n_1 - 1}\}$, $V_2 = \{b_0, b_1, \ldots, b_{n_2 - 1}\}$, and $V_3 = \{c_0, c_1, \ldots, c_{n_3 - 1}\}$. Define triangles $\Delta_{i,j} = \{a_i, b_j, c_{(i+j) \bmod n_3}\}$ for $0 \leq i \leq n_1 - 1$ and $0 \leq j \leq n_2 - 1$. By Theorem~\ref{thm1Competition}, the collection $\Gamma = \{\Delta_{i,j} : 0 \leq i \leq n_1 - 1, 0 \leq j \leq n_2 - 1\}$ forms a minimal edge clique cover of $K_{n_1,n_2,n_3}$, showing that $\theta_e(K_{n_1,n_2,n_3}) = n_1 n_2$.

\begin{theorem}[{\cite{BJGJChangCAS}}]\label{thm1Competition}
Let $n_1 \geq n_2 \geq n_3$ and $n = n_1 + n_2 + n_3$. Then the competition number of $K_{n_1,n_2,n_3}$ is
\begin{equation}\label{eq1Competition}
\kappa(K_{n_1,n_2,n_3}) = \begin{cases}
n_1 n_2 - n + 2, & \text{if } n_2 \geq n_3 + 2, \\
n_1 n_2 - n + 3, & \text{if } n_2 = n_3 + 1 \text{ or } n_2 = n_3 = 1, \\
n_1 n_2 - n + 4, & \text{if } n_2 = n_3 \geq 2.
\end{cases}
\end{equation}
\end{theorem}

For four-partite complete graphs, suppose $n_1 \geq n_2 \geq n_3 \geq n_4$ with $n_1 (n_2 - n_3 - 2) \geq n_3 (n_4 - 1)$ and $n = \sum_{i=1}^4 n_i$. Then we have $\kappa(K_{n_1, n_2, n_3, n_4}) = n_1 n_2 - n + 2$. More generally, for $r \geq 4$, if $n_1 \geq r - 2 \geq 3$, $n_1 \geq n_2 \geq n_3 \geq n_i$ for all $i \geq 4$, and
\[
n_1 (n_2 - n_3 - 2) - r + 4 \geq \theta_e(K_{n_3, n_4, \ldots, n_r}),
\]
where $n = \sum_{i=1}^r n_i$, then
\[
\kappa(K_{n_1, n_2, \ldots, n_r}) = n_1 n_2 - n + 2.
\]
In particular, Theorem~\ref{thm2Competition} implies that if $n_1 \geq 5$ and $n_1 \geq n_2 \geq n_3 + n_4 + 2$ with $n = \sum_{i=1}^5 n_i$, then
\[
\kappa(K_{n_1, n_2, n_3, n_4, n_5}) = n_1 n_2 - n + 2.
\]

\begin{theorem}[{\cite{BJGJChangCAS}}]\label{thm2Competition}
For integers $r$ and $n$ such that $2 \leq r \leq L(n) + 2$, equality $\theta_e(K_{r(n)}) = n^2$ holds if and only if $n \not\equiv 2 \pmod{4}$. In particular,
\[
\theta_e(K_{n,n,n,n}) = n^2
\]
for such $n$.
\end{theorem}

For a given vertex $u \in V(G)$, let $G_u$ denote the subgraph induced by the neighborhood of $u$. Then,
\[
\kappa(G) \geq \min_{u \in V(G)} \theta_v(G_u).
\]
Moreover, Theorem~\ref{thm3Competition} provides the following competition number bound:

\begin{theorem}[{\cite{BJGJChangCAS}}]\label{thm3Competition}
If $G$ is a graph with $n \geq r \geq 1$, then
\[
\kappa(G) \geq \min\{\theta(E(S)) : S \subseteq V(G), |S| = r\} - r + 1.
\]
\end{theorem}

For integer sequences $n_1 \geq n_2 \geq \cdots \geq n_r$ with $r \geq 2$, the connectivity of the complete multipartite graph $K_{n_1, n_2, \ldots, n_r}$ satisfies the lower bound
\begin{equation}\label{eq2Competition}
\kappa(K_{n_1, n_2, \ldots, n_r}) \geq \min\{2 n_2 - 1, n_1 + n_r - 2\}.
\end{equation}
In the special case of the balanced multipartite graph $K_{r(n)}$ with $n \geq 2$, from \eqref{eq2Competition} it follows that
\[
\kappa(K_{r(n)}) \geq 3n - 5,
\]
which can be strengthened further for $n \geq 3$ to
\[
\kappa(K_{r(n)}) \geq n^2 - r n + 3 r - 5.
\]
On the other hand, an upper bound on $\kappa(K_{r(n)})$ relates to the edge clique cover number as
\[
\kappa(K_{r(n)}) \leq \theta_e(K_{r(n)}) - 2 n + 2,
\]
which gives the explicit upper bound
\[
\kappa(K_{r(n)}) \leq n^2 - 2 n + 2
\]
for $2 \leq r \leq L(n) + 2$. These results collectively provide sharp bounds on the connectivity of complete multipartite graphs, parameterized by $r$ and $n$.

\subsection{Markov-Chain Algorithms for Generating Bipartite Graphs}
To addresses two related combinatorial problems~\cite{KannanRWTetali} involving discrete structures and Markov chains. We considers the problem of counting or generating $m \times n$ $(0,1)$-matrices with fixed row sums $r_1, r_2, \ldots, r_m$ and column sums $c_1, c_2, \ldots, c_n$, where the total sum of row sums equals the total sum of column sums, i.e., $\sum_{i=1}^m r_i = \sum_{j=1}^n c_j$. This is a classical problem in combinatorics related to contingency tables and degree sequences in bipartite graphs. It focuses on tournaments, which are complete oriented graphs, and studies the generation of tournaments with a given score sequence $\mathbf{s} = (s_1, \ldots, s_n)$, where the score sequence represents the outdegree of each vertex. The set $\mathscr{T}(\mathbf{s})$ denotes labeled tournaments with this score sequence. A fully polynomial randomized algorithm is provided to generate a member of $\mathscr{T}(\mathbf{s})$ nearly uniformly at random. The method relies on a Markov chain that is irreducible for all score sequences but currently has proven rapid mixing only for near-regular score sequences, assume $i\in \mathbb{R}$ where the scores satisfy 
\begin{equation}~\label{eqwern1}
\left| s_i - \frac{n-1}{2} \right| = O\left( n^{3/4 + \epsilon} \right),
\end{equation}

for a sufficiently small $\epsilon > 0$.

The Markov chain $(\Omega, P, \pi)$ considered is ergodic (irreducible and aperiodic) with finite state space $\Omega$, transition matrix $P$, and stationary distribution $\pi$. It is reversible, satisfying the detailed balance condition 
\[
\pi(x) P(x,y) = \pi(y) P(y,x), \quad \text{for all } x,y \in \Omega.
\] 
The total variation distance $\Delta_{tv}(t)$ measures how close the chain's distribution after $t$ steps is to stationarity, defined by 
\[
\Delta_{tv}(t) = \max_{x \in \Omega} \frac{1}{2} \sum_{y \in \Omega} \left| P^t(x,y) - \pi(y) \right|.
\]
The mixing time $\tau(\epsilon)$ is the minimal time after which the chain is within $\epsilon$-distance from stationarity for all subsequent times:
\[
\tau(\epsilon) = \min \{ t : \Delta_{tv}(t') \leq \epsilon, \forall t' \geq t \}.
\]

Conductance $\Phi$ measures the bottleneck of the chain and influences the mixing time. It is defined as 
\[
\Phi = \min_{\substack{T \subseteq \Omega \\ \pi(T) \leq 1/2}} \frac{C[T, \bar{T}]}{\pi(T)},
\]
where $\bar{T} = \Omega \setminus T$, 
\[
C[T, \bar{T}] = \sum_{\substack{x \in T \\ y \in \bar{T}}} \pi(x) P(x,y),
\]
is the weight of the cut between $T$ and $\bar{T}$, and 
\[
\pi(T) = \sum_{x \in T} \pi(x).
\]
This study links combinatorial enumeration and random generation of structured matrices and tournaments with Markov chain theory, providing efficient algorithms under specific conditions and analyzing their convergence properties through conductance and mixing times.

\section{Open Problems}
During the preparation of this research, many open problems were raised that need to be solved, and we have provided some of the necessary keys to them through the results obtained.
\begin{problem}
Assume $H M_{1}(T)>H M_{1}\left(T^{\prime}\right)$ and $H M_{2}(T)>H M_{2}\left(T^{\prime}\right)$. Then, find all possible extremal value of $H M_{1}(T)$, $H M_{1}\left(T^{\prime}\right)$, $H M_{2}(T)$ and $H M_{2}\left(T^{\prime}\right).$
\end{problem}
\begin{problem}
Find the maximum value of of $H M_{1}(T)$, $H M_{1}\left(T^{\prime}\right)$, $H M_{2}(T)$ and $H M_{2}\left(T^{\prime}\right),$ by considering all possible extremal value.
\end{problem}
\begin{problem}
Find the minimum value of of $H M_{1}(T)$, $H M_{1}\left(T^{\prime}\right)$, $H M_{2}(T)$ and $H M_{2}\left(T^{\prime}\right),$ by considering all possible extremal value.
\end{problem}
\begin{problem}
Find the maximum value of of $e^{H M_{1}}(T)$, $e^{H M_{1}}\left(T^{\prime}\right)$, $e^{H M_{2}}(T)$ and $e^{H M_{2}}\left(T^{\prime}\right),$ by considering all possible extremal value.
\end{problem}
\begin{problem}
Find the minimum value of of  $e^{H M_{1}}(T)$, $e^{H M_{1}}\left(T^{\prime}\right)$, $e^{H M_{2}}(T)$ and $e^{H M_{2}}\left(T^{\prime}\right),$ by considering all possible extremal value.
\end{problem}
\begin{problem}
According to corollary~\ref{cor1bipartite}, assume $k \geq 3$ be a positive integer and $T_0$ be a tree with bipartition $\left(Z_1, Z_2\right)$. If $G$ is a $k$-connected bipartite graph with $\delta(G) \leq \max \left\{\left|Z_1\right|,\left|Z_2\right|\right\}+k$. Prove $G$ contains a $k$-redundant tree isomorphic to $T_0$. Then, find the extremal value of $G$.
\end{problem}

\section{Conclusion}~\label{sec50000}
In this paper, we explored the structural properties of trees and bipartite graphs through the lens of topological indices, degree sequences, and combinatorial invariants. The hyper-Zagreb indices $HM_1(T)$ and $HM_2(T)$ were analyzed for trees $T \in T(n, \Delta)$, demonstrating in Propositions~\ref{pro1endsupportvertex}, \ref{pro1supportvertex}, and \ref{pro1vertex} that the presence of high-degree vertices (at least 3) distinct from a vertex of maximum degree $\Delta$ allows the construction of another tree $T' \in T(n, \Delta)$ with strictly smaller indices, i.e., $e^{HM_1(T)} > e^{HM_1(T')}$ and $e^{HM_2(T)} > e^{HM_2(T')}$. These results, supported by Lemma~\ref{rebaslem1javr}, highlight the sensitivity of hyper-Zagreb indices to vertex degree distributions in trees.

For structural characterizations, Lemma~\ref{lemtapn1} and Theorem~\ref{thmn1sec3} established that trees with distinct degree sequences exhibit ordered behavior under the $S$-order, with alternating greedy trees appearing first. Lemma~\ref{lemtapn3} further connected spectral moments to counts of $P_4$ subgraphs, providing a combinatorial interpretation of tree differences. In the context of bipartite graphs, Theorem~\ref{unmixedgraph1} characterized unmixed bipartite graphs via edge conditions, while Theorems~\ref{thm1Bipartite} and \ref{cor1bip} bounded the equitable chromatic number of trees and bipartite graphs, showing $\chi_{\mathrm{e}}(T) \leq \left\lceil \frac{|X| + |Y| + 1}{\min\{|X|,|Y|\} + 1} \right\rceil$ for trees $T(X,Y)$. Theorems~\ref{thm2Bipartite} and \ref{thm4bipartite} guaranteed long cycles and $k$-redundant tree embeddings in sufficiently connected bipartite graphs, respectively. The study of bipartite independence numbers in Theorem~\ref{thm2number} provided bounds on bi-hole sizes, with $f(n, \Delta) \geq \frac{1}{2} \frac{\log \Delta}{\Delta} n$, while Theorems~\ref{thm1covering}--\ref{thm5covering} elucidated exterior covers and dimensions of bipartite graphs $K \subseteq S \times T$, linking combinatorial and algebraic properties. Finally, the competition numbers of complete $r$-partite graphs were precisely determined in Theorems~\ref{thm1Competition} and \ref{thm2Competition}, and Markov-chain algorithms in Section~\ref{sec4} offered efficient methods for generating bipartite graphs and tournaments with specified degree or score sequences, with rapid mixing under near-regular conditions as in \eqref{eqwern1}.

These findings collectively advance our understanding of the interplay between graph structure, topological indices, and algorithmic generation, with applications in chemical graph theory, network design, and combinatorial optimization. Future work may explore extending these results to broader graph classes or refining bounds for specific degree constraints.

\section*{Declarations}
\begin{itemize}
\item Funding: Not Funding.
\item Data availability statement: All data is included within the manuscript.
\end{itemize}

\end{document}